\documentclass{amsart}
\usepackage{amsfonts,amssymb,amsmath,amsthm,bbm,xcolor,ulem}

 \newtheorem{thm}{Theorem}[section]
 \newtheorem{cor}[thm]{Corollary}
 \newtheorem{lem}[thm]{Lemma}
 \newtheorem{prop}[thm]{Proposition}
 \theoremstyle{definition}
 
 \newtheorem{rem}[thm]{Remark}
 
 \numberwithin{equation}{section}

\def\a{\mathfrak{a}}
\def\R{{\bf R}}
\def\C{{\bf C}}

\begin{document}

\begin{center}
{\bf A reduction formula for the Dunkl kernel for the root systems of type $A$}\\
P.{} Sawyer\\
Laurentian University\\
psawyer@laurentian.ca
\end{center}

\section*{Abstract}
In this paper, we provide a reduction formula for the Dunkl kernel for the root systems of type $A$. The Dunkl kernel for the root system $A_n$ is expressed as an integral involving 
the Dunkl kernel for the root system $A_{n-1}$. The corresponding reduction formula for the intertwining operator $V_k$ is given.

\section{Introduction}\label{intro}

We provide here some basic background for the paper.

A finite subset $\Sigma\subset \R^d\setminus \{0\}$ is a root system in $(\R^d,\langle\cdot,\cdot\rangle)$ if
\begin{enumerate}
\item The roots span $\R^d$.
\item The only scalar multiples of $\alpha\in \Sigma$ that are in $\Sigma$ are $\alpha$ and $-\alpha$ (the root system is reduced).
\item $\alpha$, $\beta\in \Sigma$ implies $\sigma_{\alpha}\beta=\beta-2\,\frac{\langle\alpha,\beta\rangle}{|\alpha|^2}\,\alpha\in \Sigma$.
\item $\alpha$, $\beta\in \Sigma$ implies $2\,\frac{\langle\alpha,\beta\rangle}{|\alpha|^2}$ is an integer (the root system is crystallographic).
\end{enumerate}

Given a root system $\Sigma$, one can always choose a subset $\Sigma^+$ such that $\Sigma$ is the disjoint union of $\Sigma^+$ and $-\Sigma^+$ and for $\alpha$, $\beta\in \Sigma^+$, $\alpha+\beta\in \Sigma$ implies $\alpha+\beta\in\Sigma^+$. The positive Weyl chamber is the set $\a^+=\{X\in\R^d\colon \alpha(X)>0~\hbox{for all $\alpha\in\Sigma^+$}\}$.  The Weyl group $W$ is generated by the reflections $\sigma_\alpha$, $\alpha\in\Sigma$. A function $k: \Sigma \to \R$ is called a multiplicity function if it is invariant under the action of $W$ on $\Sigma.$
	
We also introduce the building blocks of analysis on root systems. Let $\partial_\xi$ be the derivative in the direction of $\xi\in\R^d$. The Dunkl operators indexed by $\xi$ are then given by
\begin{align*}
T_\xi(k)\,f(X)&=\partial_\xi\,f(X)+\sum_{\alpha\in \Sigma_+}\,k(\alpha)\,\alpha(\xi)\,\frac{f(X)-f(\sigma_\alpha\,X)}{\langle \alpha,X\rangle}.
\end{align*}

The $T_\xi$'s, $\xi\in\R^d$, form a commutative family. For fixed $ Y\in\R^d$, the Dunkl kernel $E_k(\cdot,Y)$ is then the only real-analytic solution to the system
\begin{align*}
\left.T_\xi(k)\right\vert_X\,E_k(X,Y)
=\langle \xi,Y\rangle\, E_k(X,Y),~\forall\xi\in\R^d
\end{align*}
with $E_k(0,Y)=1$. In fact, $E_k$ extends to a holomorphic function on $\C^d\times \C^d$.  In addition, the Dunkl kernel is $W$-equivariant:
\begin{align}\label{WTW}
w\circ T_\xi(k)\circ w^{-1}=T_{w\,\xi}(k).
\end{align}

We recall some important properties of the Dunkl kernel:
\begin{enumerate}
\item $E_k(w\,X,w\,Y)=E_k(X,Y)$ for all $w\in W$.
\item $E_k(X,Y)=E_k(Y,X)$.
\item $E_k(c\,X,Y)=E_k(X,c\,Y)$ for all $c\in\C$.
\item $\overline{E_k(X,Y)}=E_k(\overline{X},\overline{Y})$.
\end{enumerate}

Another important object in Dunkl theory is the intertwining $V_k$ defined by the following properties
\begin{align*}
\left\lbrace
\begin{array}{cc}
T_\xi\,V_k&=V_k\,\partial_\xi,\\
V_k(\mathcal{P}_n)&=\mathcal{P}_n,\\
\left.V_k\right|_{\mathcal{P}_0}&=\hbox{id}
\end{array}
\right.
\end{align*}
where $\mathcal{P}_n$ is the space of homogeneous polynomials of degree $n$. We also introduce the positive measure $\mu_X$ such that
\begin{align*}
V_kf(X)=\int_{\a}\,f(H)\,d\mu_X(H)
\end{align*}
(for the existence of the positive measure, see for example \cite{Roesler2}). We have the following relation:
\begin{align*}
V_k(e^{\langle X,\cdot\rangle})=E_k(X,\cdot).
\end{align*}

For a more thorough introduction, the reader should consult \cite{Anker, Roesler1}.

In this article, we concentrate on the root systems of type $A$. The root system $A_n$ on $\{X\in \R^{n+1}\colon~x_1+\dots+x_{n+1}=0\}$ is given by $\Sigma=\{{\bf e}_i-{\bf e}_j,~i\not=j\}$ where $\{{\bf e}_1,\dots,{\bf e}_{n+1}\}$ denotes the standard basis of $\R^{n+1}$. We also use the standard scalar product $\langle\cdot,\cdot\rangle$. The Weyl group $W=S_{n+1}$ acts on $X\in\R^{n+1}$ by permuting the entries of $X$: $(w\,X)_i=X_{w(i)}$. The positive roots are $\Sigma^+=\{{\bf e}_i-{\bf e}_j,~i<j\}$.

From now on, we consider instead the root system $A_{n}$ on the space $\a=\{X=(x_1,\dots,x_{n+1})\in\R^{n+1}\}$ without the usual normalization $x_1+\dots+x_{n+1}=0$ (the first condition of a root system is then no longer satisfied). The positive Weyl chamber is $\a^+=\{(x_1,\dots,x_{n+1})\in \a\colon x_i>x_{i+1},~i=1,\dots,n\}$. We will use $\widetilde{W}$ to denote the Weyl group for $A_{n-1}$. Similarly, we will use $\widetilde{E}_k$ to denote the Dunkl kernel for $A_{n-1}$.  The multiplicity function is given by $k(\alpha)\equiv k$ (a constant parameter).  We also have
\begin{align*}
T_j(k)\,f(\lambda)=\frac{\partial f}{\partial \lambda_j}(\lambda)+k\,\sum_{i\not=j}\,\frac{f(\lambda)-f((i,j)\lambda)}{\lambda_j-\lambda_i}
\end{align*}
(so that $T_j(k)=T_{{\bf e}_j}(k)$) where $(i,j)$ exchanges the entries $\lambda_i$ and $\lambda_j$ of $\lambda$. We will assume that $k>0$ unless otherwise specified.

Let $\pi(\lambda)=\prod_{i<j}\,(\lambda_i-\lambda_j)$ and 
\begin{align*}
W_k(\lambda,\nu)=\left(\prod_{r=1}^{n+1}\left[\left(\prod_{s=r}^n(\lambda_r-\nu_s)\right)\left(\prod_{s=1}^{r-1}(\nu_s-\lambda_r)\right)\right]\right)^{k-1}.
\end{align*}

The Dunkl kernel for the root system $A_1$ is given by
\begin{align}
\lefteqn{E_k((x_1,x_2),(\lambda_1,\lambda_2))}\label{nis2}\\
&=\frac{2 \Gamma \! \left(2 k \right)}{\Gamma \! \left(k \right)^2}\,\frac{e^{x_2(\lambda_1+\lambda_2)}}{\pi(\lambda)^{2k}} \,\int_{\lambda_2}^{\lambda_1}\,e^{x_1\nu}\,e^{-x_2\nu} \,(\nu-\lambda_2)[(\lambda_1-\nu)(\nu-\lambda_2)]^{k-1}\,d\nu\nonumber
\end{align}
(here, we do not assume that $x_1+x_2=0=\lambda_1+\lambda_2$).

In Dunkl's paper \cite{Dunkl} an integral transform is given for the intertwining operator on the root system $A_2$.  
Some integral formulas for $E_k(X,\lambda)$ of the system $A_2$ are given by Amri \cite{Amri}.  In \cite{Xu}, Xu provides an integral representation of the operator $V$ for root systems of type $A$ when the function depends on only one variable.  As a consequence, they provide an expression for $E_k(X,\lambda)$ when $X={\bf e}_j$, $1\leq j\leq n$.  Later, in \cite{debie}, De Bie and Lian prove the same result using a different method.

In \cite{PGPS}, Graczyk and Sawyer adapted a formula by Amri (\cite[Theorem 1.1]{Amri}) for the Dunkl kernel for the root systems of type $A_2$. Let $\lambda\in \a^+$, the positive Weyl chamber: 
we have for any $X{\in\a}$,
\begin{align}
\lefteqn{E_k(X,\lambda)}\label{Amri1}\\
&=\frac{3\,\Gamma(3k)}{\pi(\lambda)^{2k}\Gamma(k)^3}
\,\int_{\lambda_3}^{\lambda_2} \int_{\lambda_2}^{\lambda_1} 
\{
(\nu_1-\lambda_2)(\lambda_1-\nu_2)\,\widetilde{E}_k((x_1-x_2)/2,\nu_1-\nu_2)\nonumber
\\ &\qquad\qquad\qquad\nonumber
+(\lambda_1-\nu_1)(\lambda_2-\nu_2)\,\widetilde{E}_k((x_1-x_2)/2,\nu_2-\nu_1)
\}\nonumber
\\ &\qquad\qquad\qquad 
\,(\nu_1-\lambda_3) (\nu_2-\lambda_3)
e^{(x_1+x_2-2x_3)(\nu_1+\nu_2)/2}
\,W_k(\lambda,\nu)\,d\nu_1d\nu_2\nonumber
\end{align}
assuming a normalization of $x_1+x_2+x_3=0=\lambda_1+\lambda_2+\lambda_3$.  Amri provides another expression of the Dunkl kernel in \cite[Corollary 2.2]{Amri} but 
it is the former expression that led to \eqref{Amri1}.

If we forgo the normalization, \eqref{Amri1} become
\begin{align}
\lefteqn{E_k(X,\lambda)=e^{x_3\,(\lambda_1+\lambda_2+\lambda_3)}\,\frac{3\,\Gamma(3k)}{\pi(\lambda)^{2k}\Gamma(k)^3}}\label{inspiration}\\
&\qquad\qquad\,\int_{\lambda_3}^{\lambda_2} \int_{\lambda_2}^{\lambda_1} 
\{
(\nu_1-\lambda_2)(\lambda_1-\nu_2)\,E_k^{\text{rk\,1}}((x_1,x_2),(\nu_1,\nu_2))
\nonumber\\ &\qquad\qquad\qquad
+(\lambda_1-\nu_1)(\lambda_2-\nu_2)\,E_k^{\text{rk\,1}}((x_1,x_2),(\nu_2,\nu_1))
\}\nonumber
\\ &\qquad\qquad\qquad\nonumber
\cdot (\nu_1-\lambda_3) (\nu_2-\lambda_3)
e^{-x_3(\nu_1+\nu_2)}
\,W_k(\lambda,\nu)\,d\nu_1\,d\nu_2.\nonumber
\end{align}

We now use \eqref{nis2} (we can consider $E_k(x,\nu)=e^{\nu\,x}$ as the Dunkl kernel for ``$A_0$'') and \eqref{inspiration} as an inspiration to propose a reduction formula valid for all $n$:
\begin{thm}\label{main}
Let $X\in \a$ and $\lambda\in\a^+$. Then
\begin{align}
\lefteqn{E_k(X,\lambda)=
c_{n,k}\,\frac{e^{x_{n+1}(\lambda_1+\dots+\lambda_{n+1})}}{\pi(\lambda)^{2k}} }\label{E}\\
&\qquad\qquad\,\int_{\lambda_{n+1}}^{\lambda_{n}}\,\dots\int_{\lambda_2}^{\lambda_1}
\,\sum_{w\in \widetilde{W}}\,\epsilon(w)\,Q(\lambda,w\,\nu)\,\widetilde{E}_k(X',w\,\nu)\nonumber
\\&\qquad\qquad
\cdot e^{-x_{n+1}\,(\nu_1+\dots+\nu_{n})} \,W_k(\lambda,\nu)\,d\nu_1\,\cdots\,d\nu_{n}\nonumber
\end{align}
where $c_{n,k}=\frac{(n+1)\,\Gamma((n+1)\,k)}{\Gamma(k)^{n+1}}$, $\epsilon(w)$ is the sign of the permutation $w$ and 
\begin{align}
Q(\lambda,\nu)&=\frac{\prod_{r=1}^{n+1}\left[\left(\prod_{s=r}^n(\lambda_r-\nu_s)\right)\left(\prod_{s=1}^{r-1}(\nu_s-\lambda_r)\right)\right]}{\prod_{r=1}^n\,(\lambda_r-\nu_r)}\label{P}\\
&=\prod_{r=1}^{n+1}\left[\left(\prod_{s=r+1}^{n}(\lambda_r-\nu_s)\right)\left(\prod_{s=1}^{r-1}(\nu_s-\lambda_r)\right)\right]\nonumber\\
&=(-1)^{n\,(n+1)/2}\,\prod_{\genfrac{}{}{0pt}{2}{i\not=j}{j<n+1,i\leq n+1}}\,(\lambda_i-\nu_j)\nonumber
\end{align}
and $X':=(x_1,\dots,x_n)$.
\end{thm}

This is the main result of this paper.

We will denote the domain of integration in \eqref{E} by $E(\lambda)$ and $d\nu$ for $d\nu_1\,d\nu_2\,\cdots\,d\nu_n$ when convenient.

This paper is organized as follows. In Section \ref{prelim}, some auxiliaries results are introduced as well as steps needed to simplify the proof of the main theorem. Theorem \ref{main} is then proved in Section \ref{proofofthemain}. In Section \ref{further}, we give the corresponding reduction formula for the intertwining operator $V_k$.
We show how Theorem \ref{main} provide another alternative proof of the results by Xu mentioned above.

\section{Preliminaries}\label{prelim}

To prove Theorem \ref{main}, it suffices to show that the function defined by
\begin{align}
\lefteqn{F_k(X,\lambda)=
c_{n,k}\,\frac{e^{x_{n+1}(\lambda_1+\dots+\lambda_{n+1})}}{V_k(\lambda)^{2k}} }\label{F}\\
&\qquad\cdot\int_{\lambda_{n+1}}^{\lambda_{n}}\,\dots\int_{\lambda_2}^{\lambda_1}
\,\sum_{w\in \widetilde{W}}\,\epsilon(w)\,Q(\lambda,w\,\nu)\,\widetilde{E}_k(X',w\,\nu)\nonumber
\\&\qquad \nonumber
\cdot e^{-x_{n+1}\,(\nu_1+\dots+\nu_{n})} \,W_k(\lambda,\nu)\,d\nu_1\,\cdots\,d\nu_{n}
\end{align}
satisfy the following conditions:
\begin{align}\label{conditions}
T_j(\lambda)\,F_k(X,\lambda)&=x_j\,F_k(X,\lambda),~j=1,\dots,n+1,\\
F_k(X,0)&=1~\hbox{for every $X$}\nonumber
\end{align}
where $T_j=T_{{\bf e}_j}$ and $\{{\bf e}_1, \dots, {\bf e}_{n+1}\}$ is the standard basis of $\R^{n+1}$.

\begin{rem}
In order to be able to apply the operators $\sigma_\alpha$, $\alpha\in\Sigma$, to $F_k$ defined in \eqref{F}, we have to ensure that $W_k(\lambda,\nu)$ remains
defined for $\nu\in E(\sigma_\alpha\,\lambda)$. We will consider instead
\begin{align*}
W_k(\lambda,\nu)=\left(\prod_{i=1}^n\,\prod_{j=1}^{n+1}\,(\lambda_j-\nu_i)^2\right)^{(k-1)/2}.
\end{align*}

Note that this does not change the value of $W_k(\lambda,\nu)$ when $\lambda\in\a^+$ and $\nu\in E(\lambda)$. Note also that
$\pi(\lambda)^{2\,k}$ and $W_k(\lambda,\nu)$ are symmetric (Weyl-invariant) in $\lambda$ and that $W_k(\lambda,\nu)$ is also symmetric in $\nu$.
\end{rem}

Some further properties of the Dunkl kernel:

\begin{lem}\label{winv}
Consider the Dunkl kernel on $A_{d-1}$.  We have
\begin{enumerate}
\item $T_j(Y)\,E_k(X,w\,Y)=	x_{w^{-1}(j)}\,E(X,w\,Y)$.
\item $T_j(Y)\,\left(e^{a\,(y_1+\dots+y_d)}\,E_k(X,Y)\right)=(x_j+a)\,E_k(X,Y)$.\label{zwei}
\item $e^{a\,(y_1+\dots+y_d)}\,E_k(X,Y)=E_k(X+a,Y)$ where $X+a:=(x_1+a,\dots,x_d+a)$.
\end{enumerate}
\end{lem}

\begin{proof}
~
\begin{enumerate}
\item We have
\begin{align*}
T_j(Y)\,E_k(X,w\,Y)&=T_j(Y)\,E_k(w^{-1}\,X,Y)=	x_{w^{-1}(j)}\,E(w^{-1}X,Y)\\
&=	x_{w^{-1}(j)}\,E(X,w\,Y).
\end{align*}
\item Straightforward computation.
\item Follows from \ref{zwei}.{} and $\left.e^{a\,(y_1+\dots+y_d)}\,E_k(X,Y)\right|_{Y=0}=E_k(X,0)=1$ which characterize 
$E_k(X+a,Y)$.
\end{enumerate}
\end{proof}

The next few results will be instrumental in proving that the conditions of \eqref{conditions} are satisfied.

\begin{lem}\label{S}
Let $Q$ be as in \eqref{P}. We have
\begin{enumerate}
\item If $w\in\widetilde{W}$ then $Q(w\, \lambda,\nu)=Q(\lambda,w^{-1}\, \nu)$ (the action of $w\in \widetilde{W}=S_n$ is extended to $\lambda\in \a$ by $w(n+1)=n+1$).
\item $\sum_{w\in \widetilde{W}}\,\epsilon(w)\,Q(\lambda,w\,\nu)=\prod_{i<j<n+1}\,[(\lambda_i-\lambda_j)\,(\nu_i-\nu_j)]\,\prod_{r=1}^n\,(\nu_r-\lambda_{n+1})$.
\end{enumerate}
\end{lem}

\begin{proof}
\begin{enumerate}
\item Let $w\in \widetilde{W}$: 
\begin{align*}
Q(w\, \lambda,\nu)&=(-1)^{n\,(n+1)/2}\,\prod_{\genfrac{}{}{0pt}{2}{i\not=j}{j<n+1,i\leq n+1}}\,(\lambda_{w(i)}-\nu_j)\\
&=(-1)^{n\,(n+1)/2}\,\prod_{\genfrac{}{}{0pt}{2}{i\not=j}{j<n+1,i\leq n+1}}\,(\lambda_{w(i)}-\nu_{w(w^{-1}(j))})\\
&=(-1)^{n\,(n+1)/2}\,\prod_{\genfrac{}{}{0pt}{2}{i\not=j}{j<n+1,i\leq n+1}}\,(\lambda_{i}-\nu_{w^{-1}(j)})=Q(\lambda,w^{-1}\, \nu).
\end{align*}
\item Let
\begin{align}
R(\lambda,\nu)&=\frac{\sum_{w\in \widetilde{W}}\,\epsilon(w)\,Q(\lambda,w\,\nu)}{\prod_{r=1}^n\,(\nu_r-\lambda_{n+1})}\label{R}\\
&=(-1)^{n\,(n-1)/2}\,\sum_{w\in \widetilde{W}}\,\epsilon(w)\,\prod_{i\not=j,i,j<n+1}\,(\lambda_{w(i)}-\nu_j)\nonumber
\end{align}

$R(\lambda,\nu)$ is clearly a skew-symmetric polynomial in $\lambda':=(\lambda_1,\dots, \lambda_n)$ and in $\nu$ (thanks to \eqref{R}). 
Given that $R(\lambda,\nu)$ is of degree $n\,(n-1)/2$ in $\lambda$ and also in $\nu$, we must have
$R(\lambda,\nu)=C\,\prod_{i<j<n+1}\,(\lambda_i-\lambda_j)\,(\nu_i-\nu_j)$ where $C$ is a constant. Now, let $\nu=\lambda'$ on both sides of \eqref{R}: 
\begin{align*}
R(\lambda,\lambda')
=(-1)^{n\,(n-1)/2}\,\prod_{i\not=j<n+1}\,(\lambda_i-\lambda_j)
=\prod_{i<j<n+1}\,(\lambda_i-\lambda_j)^2
\end{align*}
(all terms with $w\not=\text{id}$ are null). The result follows.

\end{enumerate}
\end{proof}

\begin{lem}\label{inv}
We consider the integrand of $F_k$ in \eqref{F}. Let 
\begin{align*}
f_k(\lambda,\nu,X)=\sum_{w\in \widetilde{W}}\,\epsilon(w)\,Q(\lambda,w\,\nu)\,\widetilde{E}_k(X',w\,\nu)
\,e^{-x_{n+1}\,(\nu_1+\dots+\nu_{n})} \,W_k(\lambda,\nu).
\end{align*}
Then
\begin{enumerate}
\item \label{zero} $f_k$ is skew-symmetric in $\nu$.
\item \label{un} If $\pi_i(D)=\pi_j(D)$ for some $i\not=j$ then $\int_D\,f_k(\lambda,\nu,X)\,d\nu=0$ (here $\pi_j\colon \R^n\to \R$ is the projection on the $j$-th component).

\item $(1,2)_\lambda \,\int_{E(\lambda)}\,f_k(\lambda,\nu,X)\,du
=-\int_{E(\lambda)}\,f_k((1,2) \,\lambda,\nu,X)\,d\nu$.
\item If $j>2$ then\\ $(1,j)_\lambda \,\int_{E(\lambda)}\,f_k(\lambda,\nu,X)\,du
=\int_{E(\lambda)}\,f_k((1,j) \,\lambda,(1,j-1) \,\nu)\,d\nu$.

\end{enumerate}
\end{lem}

\begin{proof}
Let $f_k$ be as in the statement of the lemma. 
\begin{enumerate}
\item Let $w_0\in\widetilde{W}$, using $w'=w\, w_0$,
\begin{align*}
f_k(\lambda,w_0\nu,X)&=\sum_{w\in \widetilde{W}}\,\epsilon(w)\,Q(\lambda,w\, w_0\nu)\,\widetilde{E}_k(X',w\, w_0\nu)
\,e^{-x_{n+1}\,(\nu_1+\dots+\nu_{n})} \,W_k(\lambda,\nu)\\
&=\sum_{w\in \widetilde{W}}\,\epsilon(w'w_0^{-1})\,Q(\lambda,w'\nu)\,\widetilde{E}_k(X',w'\nu)
\,e^{-x_{n+1}\,(\nu_1+\dots+\nu_{n})} \,W_k(\lambda,\nu)\\
&=\epsilon(w_0)\,f_k(\lambda,\nu,X).
\end{align*}
\item 
Suppose $\pi_i(D)=\pi_j(D)$ for $i\not=j$. Then using \eqref{zero},
\begin{align*}
\int_D\,f_k(\lambda,\nu,X)\,d\nu &=\int_{(i,j)\, D}\,f_k(\lambda,\nu,X)\,d\nu=\int_{D}\,f_k(\lambda,(i,j)\,\nu,X)\,d\nu\\
&=-\int_{D}\,f_k(\lambda,\nu,X)\,d\nu.
\end{align*}

\item We have using \eqref{un},
\begin{align*}
\lefteqn{(1,2)_\lambda \,\int_{E(\lambda)}\,f_k(\lambda,\nu,X)\,du
=\int_{E((1,2) \,\lambda)}\,f_k((1,2) \,\lambda,\nu,X)\,d\nu}\\
&=\int_{\lambda_{n+1}}^{\lambda_{n}}\int_{\lambda_{n}}^{\lambda_{n-2}} \dots\int_{\lambda_3}^{\lambda_1}\int_{\lambda_1}^{\lambda_2}f_k((1,2) \,\lambda,\nu,X)\,d\nu\\
&=\int_{\lambda_{n+1}}^{\lambda_{n}}\int_{\lambda_{n}}^{\lambda_{n-2}} \dots\left(\int_{\lambda_3}^{\lambda_2}+\int_{\lambda_2}^{\lambda_1}\right)\int_{\lambda_1}^{\lambda_2}f_k((1,2) \,\lambda,\nu,X)\,d\nu\\
&=-\int_{E(\lambda)}\,f_k((1,2) \,\lambda,\nu,X)\,d\nu.
\end{align*}

\item[(3)] If $j>2$, using \eqref{un} repeatedly, we have
\begin{align*}
\lefteqn{(1,j)_\lambda \,\int_{E(\lambda)}\,f_k(\lambda,\nu,X)\,du=\int_{E((1,j)\,\lambda)}\,f_k((1,j)\,\lambda,\nu,X)\,d\nu}\\
&=\int_{\lambda_n}^{\lambda_{n-1}}\dots
\int_{\lambda_{j+1}}^{\lambda_1}\,\int_{\lambda_1}^{\lambda_{j-1}} \dots\int_{\lambda_3}^{\lambda_2}\int_{\lambda_2}^{\lambda_{j}}f_k((1,j) \,\lambda,\nu,X)\,d\nu\\
&=\int_{\lambda_n}^{\lambda_{n-1}}\dots
\int_{\lambda_{j+1}}^{\lambda_1}\,
\left[\int_{\lambda_1}^{\lambda_2}+\int_{\lambda_2}^{\lambda_3}+\dots+\int_{\lambda_{j-2}}^{\lambda_{j-1}}\right] \dots\int_{\lambda_3}^{\lambda_2}\int_{\lambda_2}^{\lambda_{j}}f_k((1,j)\,\lambda,\nu,X)\,d\nu\\
&=\int_{\lambda_n}^{\lambda_{n-1}}\dots
\int_{\lambda_{j+1}}^{\lambda_1}\,
\int_{\lambda_1}^{\lambda_2}\dots\int_{\lambda_3}^{\lambda_2}\left[\int_{\lambda_2}^{\lambda_3}+\int_{\lambda_3}^{\lambda_4}+\dots+\int_{\lambda_{j-1}}^{\lambda_{j}}\right]f_k((1,j)\,\lambda,\nu,X)\,d\nu\\
&=\int_{\lambda_n}^{\lambda_{n-1}}\dots
\left[\int_{\lambda_{j+1}}^{\lambda_j}+\int_{\lambda_j}^{\lambda_{j-1}}+\dots+\int_{\lambda_2}^{\lambda_1}\right]\,
\int_{\lambda_1}^{\lambda_2}\dots\int_{\lambda_3}^{\lambda_2}\int_{\lambda_{j-1}}^{\lambda_{j}}f_k((1,j)\,\lambda,\nu,X)\,d\nu\\
&=\int_{\lambda_n}^{\lambda_{n-1}}\dots
\int_{\lambda_{j+1}}^{\lambda_j}\,
\int_{\lambda_1}^{\lambda_2}\dots\int_{\lambda_3}^{\lambda_2}\int_{\lambda_{j-1}}^{\lambda_{j}}f_k((1,j)\,\lambda,\nu,X)\,d\nu\\
&=\int_{E(\lambda)}\,f_k((1,j)\,\lambda,(1,j-1) \,\nu)\,d\nu.
\end{align*}

\end{enumerate}
\end{proof}

The next three results will ensure that verifying the first condition in \eqref{conditions} for $j=1$ suffices.
\begin{prop}\label{reduce}
Let $F_k(X,\lambda)$ be as in \eqref{F} and suppose $n>1$. 
Then for $1<i<n+1$, $F_k((1,i)X,(1,i)\lambda)=F_k(X,\lambda)$.
\end{prop}

\begin{proof}
Using Lemma \ref{S} and Lemma \ref{inv}, if $2<i<n+1$ then
\begin{align*}
\lefteqn{F_k((1,i)X,(1,i)\lambda)}\\
&=c_{n,k}\,\frac{e^{x_{n+1}(\lambda_1+\dots+\lambda_{n+1})}}{\pi(\lambda)^{2k}} 
\,(1,i)_\lambda\,\int_{E(\lambda)}
\,\sum_{w\in \widetilde{W}}\,\epsilon(w)\,Q(\lambda,w\,\nu)
\\&\qquad
\cdot\widetilde{E}_k((1,i)X',w\,\nu)\,e^{-x_{n+1}\,(\nu_1+\dots+\nu_{n})} \,W_k(\lambda,\nu)\,d\nu_1\,\cdots\,d\nu_{n}\\
&=c_{n,k}\,\frac{e^{x_{n+1}(\lambda_1+\dots+\lambda_{n+1})}}{\pi(\lambda)^{2k}} 
\,\,\int_{E(\lambda)}
\,\sum_{w\in \widetilde{W}}\,\epsilon(w)\,Q((1,i)\lambda,w(1,i-1)\,\nu)
\\&\qquad
\cdot\widetilde{E}_k((1,i)X',w\,(1,i-1)\nu)\,e^{-x_{n+1}\,(\nu_1+\dots+\nu_{n})} \,W_k(\lambda,\nu)\,d\nu_1\,\cdots\,d\nu_{n}\\
&=c_{n,k}\,\frac{e^{x_{n+1}(\lambda_1+\dots+\lambda_{n+1})}}{\pi(\lambda)^{2k}} 
\,\,\int_{E(\lambda)}
\,\sum_{w\in \widetilde{W}}\,\epsilon(w)\,Q(\lambda,\overbrace{(1,i)w(1,i-1)}^{w'}\,\nu)
\\&\qquad
\cdot\widetilde{E}_k((1,i)X',(1,i)(1,i)w\,(1,i-1)\nu)\,e^{-x_{n+1}\,(\nu_1+\dots+\nu_{n})} \,W_k(\lambda,\nu)\,d\nu_1\,\cdots\,d\nu_{n}\\
&=c_{n,k}\,\frac{e^{x_{n+1}(\lambda_1+\dots+\lambda_{n+1})}}{\pi(\lambda)^{2k}} 
\,\,\int_{E(\lambda)}
\,\sum_{w'\in W}\,\epsilon((1,i)w'(1,i-1))\,Q(\lambda,w'\,\nu)
\\&\qquad
\cdot\widetilde{E}_k((1,i)X',(1,i)w'\nu)\,e^{-x_{n+1}\,(\nu_1+\dots+\nu_{n})} \,W_k(\lambda,\nu)\,d\nu_1\,\cdots\,d\nu_{n}
=F_k(X,\lambda)
\end{align*}
since $\widetilde{E}_k((1,i)X',(1,i)w'\nu)=\widetilde{E}_k(X',w'\nu)$.

If $i=2$, the proof is similar with a minus sign in front of the integrals and with $(1,i-1)$ replaced by the identity. The minus sign later disappears at the end since then $\epsilon((1,i)w'(1,i-1))=-\epsilon(w')$.
\end{proof}

\begin{lem}\label{summation}
We have $\sum_{i=1}^{n+1}\,T_i(\lambda)=\sum_{i=1}^{n+1}\,\frac{\partial~}{\partial \lambda_i}$.
\end{lem}

\begin{proof}
This is a simple verification (the reflection terms cancel out).
\end{proof}

\begin{lem}\label{summation2}
Let $F_k(\lambda,X)$ be as in Proposition \ref{reduce}. Then
\begin{align*}
\sum_{i=1}^{n+1}\,T_i(\lambda)\,F_k(\lambda,X)=\sum_{i=1}^{n+1}\,x_i\,F_k(\lambda,X).
\end{align*}
\end{lem}

\begin{proof}
We will write $\lambda+t:=(\lambda_1+t,\dots,\lambda_{n+1}+t)$. By Lemma \ref{summation} and using $\nu_k'=\nu_k-t$, 
\begin{align*}
\lefteqn{\sum_{i=1}^{n+1}\,T_i(\lambda)\,F_k(\lambda,X)
=\sum_{i=1}^{n+1}\,\frac{\partial~}{\partial \lambda_i}\,F_k(\lambda,X)
=\left.\frac{d~}{dt}\right|_{t=0}\,F_k(\lambda+t,X)}\\
&=c_{n,k}\,\left.\frac{d~}{dt}\right|_{t=0}
\,\frac{e^{x_{n+1}(\lambda_1+\dots+\lambda_{n+1})+(n+1)\,t\,x_{n+1}}}{\pi(\lambda)^{2k}}\\
&\qquad\cdot\int_{\lambda_{n+1}+t}^{\lambda_{n}+t}\,\dots\int_{\lambda_2+t}^{\lambda_1+t}
\,\sum_{w\in \widetilde{W}}\,\epsilon(w)\,Q(\lambda+t,w\,\nu)\,\widetilde{E}_k(X',w\,\nu)
\\&\qquad
\cdot e^{-x_{n+1}\,(\nu_1+\dots+\nu_{n})} \,W_k(\lambda+t,\nu)\,d\nu_1\,\cdots\,d\nu_{n}\\
&=c_{n,k}\,\left.\frac{d~}{dt}\right|_{t=0}
\,\frac{e^{x_{n+1}(\lambda_1+\dots+\lambda_{n+1})+(n+1)\,t\,x_{n+1}}}{\pi(\lambda)^{2k}} 
\\&\qquad
\cdot\,\int_{E(\lambda)}\,\sum_{w\in \widetilde{W}}\,\epsilon(w)\,\overbrace{Q(\lambda+t,w\,(\nu'+t))}^{Q(\lambda,w\,\nu')}
\,\widetilde{E}_k(X',w\,(\nu'+t))
\\&\qquad
\cdot e^{-x_{n+1}\,(\nu_1'+\dots+\nu_{n}')-n\,t\,x_{n+1}}\,
\overbrace{W_k(\lambda+t,\nu'+t)}^{W_k(\lambda,\nu')}\,d\nu_1'\,\cdots\,d\nu_{n}'\\
&=c_{n,k}\,\left.\frac{d~}{dt}\right|_{t=0}
\,\frac{e^{x_{n+1}(\lambda_1+\dots+\lambda_{n+1})+t\,x_{n+1}}}{\pi(\lambda)^{2k}} 
\,\,\int_{E(\lambda)}
\,\sum_{w\in \widetilde{W}}\,\epsilon(w)\,Q(\lambda,w\,\nu')
\\&\qquad
\cdot\widetilde{E}_k(X',w\,(\nu'+t))\,e^{-x_{n+1}\,(\nu_1'+\dots+\nu_{n}')} W_k(\lambda,\nu')\,d\nu_1'\,\cdots\,d\nu_{n}'\\
&=\sum_{i=1}^{n+1}\,x_i\,F_k(\lambda,X)
\end{align*}
using the fact that 
\begin{align*}
\left.\frac{d~}{dt}\right|_{t=0}\,\widetilde{E}_k(X',w\,(\nu'+t))
&=\sum_{i=1}^{n}\,\frac{\partial~}{\partial\nu'_i}\,\widetilde{E}_k(X',w\,\nu')
=\sum_{i=1}^{n}\,T_i(\nu')\,\widetilde{E}_k(X',w\,\nu')\\
&=\sum_{i=1}^{n}\,x_i\,E_k(X',w\,\nu').
\end{align*}
\end{proof}

\section{Proof of the main result}\label{proofofthemain}

\begin{proof}[{\bf (Proof of Theorem \ref{main})}]
We will show that $F_k(X,\lambda)$ is the Dunkl kernel by showing it has the right properties thus proving Theorem \ref{main}.
We start by showing that the second condition of \eqref{conditions} is satisfied.

First, we have using using Lemma \ref{S},
\begin{align*}
F_k(0,\lambda)
&=\frac{c_{n,k}}{\pi(\lambda)^{2k}}
\,\,\int_{E(\lambda)}
\,\sum_{w\in \widetilde{W}}\,\epsilon(w)\,Q(\lambda,w\,\nu) \,
W_k(\lambda,\nu)\,d\nu_1\,\cdots\,d\nu_{n}\\
&=c_{n,k}\,\frac{\pi(\lambda')\,\prod_{i=1}^{n}\,(\lambda_i-\lambda_{n+1})}{\pi(\lambda)^{2\,k}}
\,\,\int_{E(\lambda)}
\,\pi(\nu)
\\&\qquad
 \cdot\prod_{j=1}^{n}\frac{\nu_j-\lambda_{n+1}}{\lambda_i-\lambda_{n+1}}
\left(\pi(\lambda)^2\,\prod_{p=1}^{n+1}\,\frac{\prod_{i=1}^{n}(\nu_i-\lambda_p)}{\prod_{i\not=p}(\lambda_i-\lambda_p)}\right)^{k-1}
\,d\nu_1\,\cdots\,d\nu_{n}.
\end{align*}

We use the change of variables (also used in \cite{Sawyer}) 
\begin{align}
t_p=\frac{\prod_{i=1}^{n}(\nu_i-\lambda_p)}{\prod_{i\not=p}(\lambda_i-\lambda_p)}, ~\hbox{$1\leq p\leq n+1$},\label{changet}
\end{align}
noting that $t_p\geq0$ for all $p$, $\sum_{p=1}^{n+1}\,t_p=1$ and setting $\sigma=\{(t_1,\dots,t_{n+1})\colon t_p\geq 0,~ \sum_{p=1}^{n+1}\,t_p=1\}$. We have
\begin{align*}
\frac{D(t_i)}{D(\nu _j)} &= \frac{\prod _{i<p\leq n}(\nu _i-\nu _p)}{\prod
_{i<p\leq n+1}(\lambda _i-\lambda _p)},
\end{align*}
and therefore
\begin{align*}
F_k(0,\lambda)&=c_{n,k}\,\frac{\pi(\lambda')\,\prod_{i=1}^{n}\,(\lambda_i-\lambda_{n+1})}{\pi(\lambda)^{2\,k}}
\,\pi(\lambda)\,\pi(\lambda)^{2\,k-2}\,\int_\sigma\,t_{n+1}\,(t_1\,\cdots t_{n+1})^{k-1}\,dt\\
&=c_{n,k}\,\int_\sigma\,t_{n+1}\,(t_1\,\cdots t_{n+1})^{k-1}\,dt=1.
\end{align*}
The last equality follows from the integration of the Dirichlet distribution of order $n+1$ with parameters $k$, \dots, $k$ (repeated $n$ times) and $k+1$.

We now show that $T_1(\lambda)\,F_k(\lambda,X)=x_1\,F_k(\lambda,X)$. We assume first that $\Re k$ is ``large enough'' to ensure that the integrand and its derivatives are zero on the domain of integration. Afterward, we extend the result to $\Re k>0$ by analytic continuation on $k$.

Let $C_i=\frac{\partial~}{\partial \lambda_1}+\frac{\partial~}{\partial \nu_i}$ and $D_i=\frac{\partial~}{\partial \nu_i}$, $i=1$, \dots, $n$. Let 
\begin{align*}
H_k(X,\lambda)=e^{-x_{n+1}\,(\lambda_1+\dots+\lambda_{n+1})}\,\frac{\pi(\lambda)^{2k}}{c_{n,k}}\,F_k(X,\lambda)
\end{align*}
and $\widetilde{H}_k(X,\nu)=e^{-x_{n+1}\,(\nu_1+\dots+\nu_{n+1})}\,\widetilde{E}_k(X',\nu)$.

We have
\begin{align*}
\lefteqn{e^{-x_{n+1}\,(\lambda_1+\dots+\lambda_{n+1})}\,\frac{\pi(\lambda)^{2k}}{c_{n,k}}\,T_1\,F_k(X,\lambda)}\\
&=
\overbrace{-2\,k\left(\frac1{\lambda_1-\lambda_2}+\dots+\frac1{\lambda_1-\lambda_{n+1}}\right)\,H_k(X,\lambda)}^{B_1}
+\overbrace{x_{n+1}\,H_k(X,\lambda)}^{L_1}
\\&\qquad
+\overbrace{\int_{E(\lambda)}
\sum_{i=1}^{n}
\,\sum_{w\in \widetilde{W},w(1)=i}\,\epsilon(w) \,
\frac{C_i\,[Q(\lambda,w\nu)\,W_k(\lambda,\nu)]\,\widetilde{H}_k(X',w\,\nu)}{W_k(\lambda,\nu)}\,W_k(\lambda,\nu) d\nu}^{B_2}
\\&\qquad
+\overbrace{
\sum_{i=1}^{n}
\,\sum_{w\in \widetilde{W},w(1)=i}\,
\epsilon(w) \,\left[\int_{E(\lambda)}\,\left((-D_i)[Q(\lambda,w\nu)\,W_k(\lambda,\nu)]\,\widetilde{H}_k(X',w\,\nu)\right)\,d\nu+S_{w,i}\right]}^{L_2}
\\&\qquad
+\overbrace{\frac{k}{\lambda_1-\lambda_2}
\,\int_{E(\lambda)}
\sum_{w\in \widetilde{W}}\,\epsilon(w) \,
(Q(\lambda,w\nu)+Q((1,2)\lambda,w\nu))\,\widetilde{H}_k(X',w\,\nu)
\,W_k(\lambda,\nu)\,d\nu}^{B_3}
\\&\qquad
+\sum_{i=3}^{n+1}\,\frac{k}{\lambda_1-\lambda_i}\,
\,\int_{E(\lambda)}
\sum_{w\in \widetilde{W}}\,\epsilon(w) \,
\{
Q(\lambda,w\nu)\,\widetilde{H}_k(X',w\,\nu)
\\&\qquad
-Q((1,i)\lambda,w(1,i-1)\nu)\,\widetilde{H}_k(X,w(1,i-1)\,\nu)\}\,W_k(\lambda,\nu)\,d\nu~\fbox{\hbox{$B_4$}}\\
&\overbrace{-\epsilon(w)\,\sum_{i=1}^{n} \,\sum_{w\in \widetilde{W},w(1)=i}\,S_{w,i}}^S
\end{align*}
where
\begin{align*}
S_{w,i}&=k\,\sum_{j\not=i<n+1}\,\int_{E(\lambda)}
\,Q(\lambda,w\nu)\,\frac{\widetilde{H}_k(X',w\nu)-\widetilde{H}_k(X',w\,(i,j)\,\nu)}{\nu_{i}-\nu_{j}}
\,W_k(\lambda,\nu)\,d\nu.
\end{align*}

Note that  $L_2=(x_1-x_{n+1})\,H_k(X,\lambda)$. This equality follows from taking the adjoints of the $-D_i$'s, taking into account the terms $S_{w,i}$ and Lemma \ref{winv}.

We want to show that $B_1+B_2+B_3+B_4+S=0$. The integrand of $e^{-x_{n+1}\,(\lambda_1+\dots+\lambda_{n+1})}\,\frac{\pi(\lambda)^{2k}}{c_{n,k}}\,T_1\,F_k(X,\lambda)$ is skew-symmetric in $\nu$. Thus it suffices to show that the coefficient of $k\,\widetilde{H}_k(X',\nu)\,W_k(\lambda,\nu)$ in the integrands of $B_1$, $B_2$, $B_3$, $B_4$ and $S$ (we will denote them respectively $b_1$, $b_2$, $b_3$, $b_4$ and $s$) add up to 0.

Now,
\begin{align*}
b_1&=-2\,\left(\frac1{\lambda_1-\lambda_2}+\dots+\frac1{\lambda_1-\lambda_{n+1}}\right)\,Q(\lambda,\nu)\\
b_2&=\frac{C_1\,[Q(\lambda,\nu)\,W_k(\lambda,\nu)]}{k\,W_k(\lambda,\nu)}
=Q(\lambda,\nu)\,\left(\sum_{i=2}^{n}\,\frac1{\lambda_1-\nu_i}+\sum_{i=2}^{n}\,\frac1{\nu_1-\lambda_i}\right),\\
b_3&=\frac{Q(\lambda,\nu)+Q((1,2)\lambda,\nu)}{\lambda_1-\lambda_2},\\
b_4&=\sum_{i=3}^{n+1}\,\frac1{\lambda_1-\lambda_i}\,(Q(\lambda,\nu)+Q((1,i)\lambda,\nu)),\\
s&=-\sum_{j=2}^{n}\,\frac{Q(\lambda,\nu)-Q(\lambda,(1,j)\nu)}{\nu_1-\nu_j}.
\end{align*}

To compute $b_2$, it is helpful to notice that 
\begin{align}
Q(\lambda,\nu)\,W_k(\lambda,\nu)=W_{k+1}(\lambda,\nu)/\prod_{j=1}^n\,(\lambda_j-\nu_j).\label{QW}
\end{align}

Hence,
\begin{align*}
\lefteqn{b_1+b_2+b_3+b_4+s}\\
&=\sum_{i=2}^{n+1}\,\frac{Q((1,i)\lambda,\nu)-Q(\lambda,\nu)}{\lambda_1-\lambda_i}
+Q(\lambda,\nu)\,\left(\sum_{i=2}^{n}\,\frac1{\lambda_1-\nu_i}+\sum_{i=2}^{n+1}\,\frac1{\nu_1-\lambda_i}\right)
\\&\qquad
-\sum_{j=2}^{n}\,\frac{Q(\lambda,\nu)-Q(\lambda,(1,j)\nu)}{\nu_1-\nu_j}\\
&=Q(\lambda,\nu)\,\left(-\sum_{i=2}^{n}\,\frac{\nu_1-\nu_i}{(\lambda_1-\nu_i)\,(\nu_1-\lambda_i)}
-\frac1{\nu_1-\lambda_{n+1}}
\right.\\&\qquad\left.
+\sum_{i=2}^{n}\,\frac1{\lambda_1-\nu_i}+\sum_{i=2}^{n+1}\,\frac1{\nu_1-\lambda_i}
-\sum_{i=2}^{n}\,\frac{\lambda_1-\lambda_i}{(\lambda_1-\nu_i)\,(\nu_1-\lambda_i)}\right)=0
\end{align*}
using the fact that 
\begin{align*}
Q((1,i)\lambda,\nu)-Q(\lambda,\nu)&= Q(\lambda,(1,i)\nu)-Q(\lambda,\nu)\\
&=-Q(\lambda,\nu)\,\frac{(\lambda_1-\lambda_i)\,(\nu_1-\nu_i)}{(\lambda_1-\nu_i)\,(\nu_1-\lambda_i)}~\hbox{when $i<n+1$},\\
Q((1,n+1)\lambda,\nu)-Q(\lambda,\nu)&=-Q(\lambda,\nu)\,\frac{\lambda_1-\lambda_{n+1}}{\nu_1-\lambda_{n+1}}.
\end{align*}

Next, we show that $T_j(\lambda)\,F_k(\lambda,X)=x_j\,F_k(\lambda,X)$ for $1<j<n+1$:
it suffices to use the above along \eqref{WTW} and Proposition \ref{reduce}:\\
$F_k(\lambda,X)=F_k((1,j)\,X,(1,j)\,\lambda)$.

Finally, to show that $T_{n+1}(\lambda)\,F_k(\lambda,X)=x_{n+1}\,F_k(\lambda,X)$, it suffices to observe that
$T_{n+1}=(T_1+\dots+T_{n+1})-T_{n+1}$ and use Lemma \ref{summation}.
\end{proof}

\section{Further results and conclusion}\label{further}

The following result will serve as a tool to allow us to provide a reduction formula for the operator $V_k$.

\begin{thm}[Nachbin \cite{Nachbin}]\label{Nach}
Let $A$ be a subalgebra of the algebra $C^\infty(M)$ of smooth functions on a finite dimensional smooth manifold $M$. Suppose that $A$ separates the points of $M$ and also separates the tangent vectors of $M$: for each point $m \in M$ and tangent vector $v$ at the tangent space at $m$, there is a $f \in A$ such that $df(x)(v) \not= 0$. Then $A$ is dense in $C^\infty(M)$.
\end{thm}

\begin{thm}\label{second}
Let $f(\lambda)=f(\lambda_1,\dots,\lambda_{n+1})$ be a smooth function and let
\begin{align*}
f_\lambda(\nu_1,\dots,\nu_n)=f(\nu_1,\dots,\nu_n,\sum_{i=1}^{n+1}\,\lambda_i-\sum_{i=1}^n\,\nu_i).
\end{align*}
Then
\begin{align}
V_k(f)(\lambda)
&=
\frac{c_{n,k}}{\pi(\lambda)^{2k}} 
\,\int_{E(\lambda)}\label{V}
\,\sum_{w\in W}\,\epsilon(w)\,Q(\lambda,w\,\nu)\,\widetilde{V_k}(f_\lambda)(w\,\nu)
\\&\qquad\qquad\qquad\qquad\cdot W_k(\lambda,\nu)\,d\nu_1\cdots\,d\nu_{n}\nonumber
\end{align}
where $\widetilde{V_k}$ is the corresponding operator on $A_{n-1}$.
\end{thm}

\begin{proof}
Using Lemma \ref{winv}, we can write
\begin{align*}
\lefteqn{V_k(\exp(\langle \cdot,X\rangle))(\lambda)=E_k(X,\lambda)=c_{n,k}\,\frac{e^{x_{n+1}(\lambda_1+\dots+\lambda_{n+1})}}{\pi(\lambda)^{2k}} }
\\&\qquad
\cdot\,\int_{E(\lambda)}
\,\sum_{w\in W}\,\epsilon(w)\,Q(\lambda,w\,\nu)\,\widetilde{E}_k(X'-x_{n+1},w\,\nu)
\,W_k(\lambda,\nu)\,d\nu_1\cdots\,d\nu_{n}.
\end{align*}

Now, if $f(\lambda)=\exp(\langle \lambda,X\rangle)$ then 
\begin{align*}
\widetilde{V}_k(f_\lambda(\nu))&=\widetilde{V}_k(\exp(\sum_{i=1}^n\,x_i\,\nu_i+x_{n+1}\,(\sum_{i=1}^{n+1}\,\lambda_i-\sum_{i=1}^n\,\nu_i)))\\
&=e^{x_{n+1}\,(\lambda_1+\dots+\lambda_{n+1})}\,\widetilde{V}_k(\exp(\sum_{i=1}^n\,(x_i-x_{n+1})\,\nu_i))\\
&=e^{x_{n+1}\,(\lambda_1+\dots+\lambda_{n+1})}\,\widetilde{E}_k(X'-x_{n+1},\nu)
\end{align*}
and equation \eqref{V} holds. Therefore, \eqref{V} also holds for any linear combination of exponentials of the form $\exp(\langle X,\cdot\rangle)$. Theorem \ref{Nach} allows us to conclude.
\end{proof}

\begin{rem}
We can compare the expression in Theorem \ref{second} for the Dunkl kernel and the intertwining operator $V_k$ with its Weyl-invariant counterparts, the spherical function $E^W(X,\lambda)$ and the dual Abel transform $\mathcal{A}^*$, in \cite{Sawyer0}:

\begin{align*}
E^W(X,\lambda)&=\frac{1}{|W|}\,\sum_{w\in W}\,E_k(X,w\,\lambda)=\frac{1}{|W|}\,\sum_{w\in W}\,E_k(w\,X,\lambda)\\
&=\frac{\Gamma(k\,(n+1))}{(\Gamma(k))^{n+1}}\,\frac{e^{x_{n+1}\,(\lambda_1+\dots+\lambda_{n+1})}}{\pi(\lambda)^{2\,k-1}}\,\int_{E(\lambda)}
\,\widetilde{E^W}(X'-x_{n+1},\nu)
\\&\qquad\qquad\qquad\qquad\qquad\qquad\qquad
\cdot\pi(\nu)\,W_k(\lambda,\nu)\,d\nu,\\
\mathcal{A}^*(f)(\lambda)	
&=\frac{\Gamma(k\,(n+1))}{(\Gamma(k))^{n+1}}\,\frac{1}{\pi(\lambda)^{2\,k-1}}\,\int_{E(\lambda)}
\,\widetilde{\mathcal{A}^*}(f_\lambda)(\nu)	\,\pi(\nu)\,W_k(\lambda,\nu)\,d\nu.
\end{align*}

It was shown in \cite{Sawyer0}, for the root systems of type $A$,  that the supports of $\mathcal{A}^*$ and of $V_k$ are the same, namely $C(\lambda)$ (the convex hull of $W\,\lambda$ in $\a$).
\end{rem}

We now use equation Theorem \ref{main} to provide a different proof of a result by Xu (\cite[Corollary 2.4]{Xu}).
\begin{thm}\label{debie0}
\begin{align}
E_k(x_{n+1}\,{\bf e}_{n+1})&=E_k((0,\dots,x_{n+1}),\lambda)
\label{otherproof}\\
&=c_{n,k}\,\int_\sigma\,e^{-x_{n+1}\,(\lambda_1\,t_1+\dots+\lambda_{n+1}\,t_{n+1})}\,t_{n+1}\,(t_1\,\cdots t_{n+1})^{k-1}\,dt.\nonumber
\end{align}
\end{thm}

\begin{proof}
From Theorem \ref{main} and Lemma \ref{S},
\begin{align*}
\lefteqn{E_k((0,\dots,x_{n+1}),\lambda)=c_{n,k}\,e^{-x_{n+1}\,(\lambda_1+\dots+\lambda_{n+1})}\,\frac{\pi(\lambda')\,\prod_{i=1}^{n}\,(\lambda_i-\lambda_{n+1})}{\pi(\lambda)^{2\,k}}}\\
&\qquad\cdot\int_{E(\lambda)}
\,\pi(\nu)
 \,\prod_{j=1}^{n}\frac{\nu_j-\lambda_{n+1}}{\lambda_j-\lambda_{n+1}}
\left(\pi(\lambda)^2\,\prod_{p=1}^{n+1}\,\frac{\prod_{i=1}^{n}(\nu_i-\lambda_p)}{\prod_{i\not=p}(\lambda_i-\lambda_p)}\right)^{k-1}
\\&\qquad
\cdot e^{-x_{n+1}\,(\nu_1+\dots+\nu_N)}\,W_k(\lambda,\nu)\,d\nu_1\,\cdots\,d\nu_n.
\end{align*}
Using the change of variable given in \eqref{changet}, we obtain \eqref{otherproof}.  We have used the fact that $\sum_{j=1}^{n+1}\,\lambda_j\,t_j=\sum_{j=1}^{n+1}\,\lambda_j-\sum_{j=1}^{n}\,\nu_j$. Indeed, for $n=1$, this is easy to check. We then assume true for $n-1$, $n\geq 2$ and
consider the sum as a function of $\nu_n$, 
\begin{align*}
S(\nu_n)=\sum_{p=1}^{n+1}\,\lambda_p\,\frac{\prod_{i=1}^{n}(\nu_i-\lambda_p)}{\prod_{i\not=p}(\lambda_i-\lambda_p)}.
\end{align*}

Observe that $S$ is a polynomial of degree 1 in $\nu_n$. Using induction, we note that
\begin{align*}
S(\lambda_{n+1})&=\sum_{p=1}^{n}\,\lambda_p\,\frac{\prod_{i=1}^{n-1}(\nu_i-\lambda_p)}{\prod_{i\not=p}(\lambda_i-\lambda_p)}=\sum_{i=1}^n\,\lambda_i-\sum_{i=1}^{n-1}\,\nu_i,\\
S(\lambda_1)&=\sum_{p=2}^{n+1}\,\lambda_p\,\frac{\prod_{i=2}^{n-1}(\nu_i-\lambda_p)}{\prod_{1<i\not=p}(\lambda_i-\lambda_p)}=\sum_{i=2}^{n+1}\,\lambda_i-\sum_{i=1}^{n-1}\,\nu_i.
\end{align*}

The only polynomial in $\nu_n$ of degree 1 satisfying these two relations is $S(\nu_n)=\sum_{i=1}^{n+1}\,\lambda_i-\sum_{i=1}^{n}\,\nu_i$.
\end{proof}

\begin{cor}
Let $1\leq j\leq n+1$. Then 
\begin{align*}
E_k(x_j\,{\bf e}_j,\lambda)
=c_{n,k}
\,\int_\sigma\,e^{-x_j\,(\lambda_1\,t_1+\dots+\lambda_{n+1}\,t_{n+1})}\,t_j\,(t_1\,\cdots t_{n+1})^{k-1}\,dt.
\end{align*}
\end{cor}

\begin{proof}
We have 
\begin{align*}
E_k(x_j\,{\bf e}_j,\lambda)&=E_k(x_j\,(j,n+1)\,{\bf e}_j,(j,n+1)\,\lambda)\\
&=c_{n,k}
\,\int_\sigma\,e^{-x_j\,(\lambda_1\,t_1+\dots+\lambda_{n+1}\,t_j+\dots+\lambda_j\,t_{n+1})}\,t_{n+1}\,(t_1\,\cdots t_{n+1})^{k-1}\,dt
\\
&=c_{n,k}
\,\int_\sigma\,e^{-x_j\,(\lambda_1\,t_1+\dots+\lambda_{n+1}\,t_{n+1})}\,t_j\,(t_1\,\cdots t_{n+1})^{k-1}\,dt
\end{align*}
(exchanging $t_j$ and $t_{n+1}$).
\end{proof}

Similarly, Theorem \ref{second} leads to a proof of \cite[Theorem 2.1]{Xu} of Xu.

\begin{thm}\label{debie1}
Let $f\colon \R\to \R$ and define $F$ by $F(\lambda_1,\dots,\lambda_{n+1})=f(\lambda_j)$. Then 
\begin{align*}
V_k(F)(\lambda)&=c_{n,k}
\,\int_\sigma\,f(\lambda_1\,t_1+\dots+\lambda_{n+1}\,t_{n+1})\,t_j\,(t_1\,\cdots t_{n+1})^{k-1}\,dt.
\end{align*}
\end{thm}

\begin{proof}
A combination of Theorem \ref{second} and the proof of Theorem \ref{debie0} and its corollary.
\end{proof}

We can write the formulas of Theorem \ref{main} and Theorem \ref{second} more compactly using \eqref{QW}:
\begin{align}
E_k(X,\lambda)&=
c_{n,k}\,\frac{e^{x_{n+1}(\lambda_1+\dots+\lambda_{n+1})}}{\pi(\lambda)^{2k}}\label{problem}
\\&\qquad
\cdot\,\int_{E(\lambda)}
\,\sum_{w\in \widetilde{W}}\,\epsilon(w)\,\frac{\widetilde{E}_k(X'-x_{n+1},w\,\nu)}{\prod_{i=1}^n\,(\lambda_i-\nu_{w(i)})}
\,W_{k+1}(\lambda,\nu)\,d\nu,\nonumber\\
V_k(f)(\lambda)
&=
\frac{c_{n,k}}{\pi(\lambda)^{2k}} 
\,\,\int_{E(\lambda)}\label{V}
\,\sum_{w\in \widetilde{W}}\,\epsilon(w)\,\frac{\widetilde{V_k}(f_\lambda)(w\,\nu)}{\prod_{i=1}^n\,(\lambda_i-\nu_{w(i)})}
\, W_{k+1}(\lambda,\nu)\,d\nu.\nonumber
\end{align}

\begin{rem}
One drawback of Theorem \ref{main} is that when $n>2$, not all terms in the sum in the integrand of \eqref{problem} are positive. Indeed, if $w$ is the permutation sending 1, 2, 3 to 2, 3, 1 respectively and $i>3$ to $i$ itself, then $\epsilon(w)=1$ and 
\begin{align*}
\prod_{i=1}^n\,(\lambda_i-\nu_{w(i)})=(\lambda_1-\nu_2)\,(\lambda_2-\nu_3)\,(\lambda_3-\nu_1)\,\prod_{i=4}^n\,(\lambda_i-\nu_i)
\end{align*}
which is negative on $E(\lambda)^\circ$.
\end{rem}

\end{document}